\documentclass[11pt]{amsart}
\usepackage{amsmath,amssymb,amsthm}
\usepackage[latin1]{inputenc}
\usepackage{version,tabularx,multicol}
\usepackage{graphicx,float,psfrag}
 \usepackage{stmaryrd}

\newcommand{\esp}{\vspace{.2cm}}

\newcommand{\reff}[1]{(\ref{#1})}

\theoremstyle{plain}
\newtheorem{theo}{Theorem}[section]
\newtheorem{cor}[theo]{Corollary}
\newtheorem{prop}[theo]{Proposition}
\newtheorem{lem}[theo]{Lemma}

\theoremstyle{remark}
\newtheorem{rem}[theo]{Remark}

\newcommand{\ca}{{\mathcal A}}

\newcommand{\cf}{{\mathcal F}}
\newcommand{\cg}{{\mathcal G}}

\newcommand{\cl}{{\mathcal L}}

\newcommand{\cs}{{\mathcal S}}

\newcommand{\cu}{{\mathcal U}}

\newcommand{\A}{{\mathbb A}}

\newcommand{\E}{{\mathbb E}}

\newcommand{\N}{{\mathbb N}}
\renewcommand{\P}{{\mathbb P}}

\newcommand{\R}{{\mathbb R}}

\newcommand{\T}{{\mathbb T}}
\newcommand{\Z}{{\mathbb Z}}

\newcommand{\bt}{{\mathbf t}}
\newcommand{\bs}{{\mathbf s}}

\newcommand{\ind}{{\bf 1}}

\newcommand{\Card}{{\rm Card}\;}

\newcommand{\dist}{{\rm dist}\;}

\newcommand{\inv}[1]{\mathop{\frac{1}{ #1}}\nolimits}
\newcommand{\expp}[1]{\mathop {\mathrm{e}^{ #1}}}

\title[Conditioned Galton-Watson trees]{Local limits of conditioned Galton-Watson trees I: the infinite spine case}

\date{\today}
\author{Romain Abraham} 

\address{
Romain Abraham,
Laboratoire MAPMO, CNRS, UMR 7349,
F\'ed\'eration Denis Poisson, FR 2964,
 Université d'Orléans,
B.P. 6759,
45067 Orléans cedex 2,
France.
}
  
\email{romain.abraham@univ-orleans.fr}

\author{Jean-François Delmas}

\address{
Jean-Fran\c cois Delmas,
Université Paris-Est, \'Ecole des Ponts, CERMICS, 6-8
av. Blaise Pascal, 
  Champs-sur-Marne, 77455 Marne La Vallée, France.}

\email{delmas@cermics.enpc.fr}

\begin{document}

\begin{abstract}
  We give  a necessary and  sufficient condition for the  convergence in
  distribution   of  a  conditioned   Galton-Watson  tree   to  Kesten's
  tree.  This yields  elementary proofs  of Kesten's  result as  well as
  other  known  results on  local  limits  of conditioned  Galton-Watson
  trees.  We  then apply  this  condition to  get  new  results in  the
  critical  case (with  a general  offspring  distribution)  and  in  the
  sub-critical cases (with a generic offspring distribution)  on the limit
  in distribution of a Galton-Watson  tree conditioned on having a large
  number of individuals with out-degree in a given set.
\end{abstract}

\keywords{Galton-Watson, random tree, local-limit, non-extinction, 
branching process}

\subjclass[2010]{60J80, 60B10}

\maketitle

\section{Introduction}

Galton-Watson  (GW)   processes  constitute  a  very   simple  model  of
population growth where all individuals give birth independently of each
others  to  a  random  number   of  children  with  the  same  offspring
distribution  $p$.   This  population  growth  can  be  described  by  a
genealogical tree $\tau$ that we call the GW tree. It is well-known that
in  the sub-critical  case  (the mean  number  of children  of a  single
individual is strictly less than 1) and in the (non-degenerate) critical
case (the mean number of children  of an individual is 1) the population
becomes a.s.  extinct.  However,  one can define  in these two  cases a
tree $\tau^*$ with an infinite spine, that we call Kesten's tree in this
paper,  which can  be seen  as the  tree conditioned  on non-extinction,
defined  as  the  local  limit   in  distribution  of  the  tree  $\tau$
conditioned to  reach height $n$, when  $n$ tends to  infinity, see {\sc
  Kesten}  \cite{k:sbrwrc}.  This result  is  recalled  here in  Section
\ref{sec:kesten-main}. The  tree $\tau^*$ happens to  be the size-biased
tree  already studied  earlier, see  e.g.  {\sc  Hawkes} \cite{h:tgsbp},
{\sc Joffe}  and {\sc Waugh}  \cite{jw:edkngwp} as well as  {\sc Lyons},
{\sc Pemantle} and {\sc Peres} \cite{lpp:cpLlogLcmbbp}.  It also appears
(for GW processes  only) as a Q-process  and can be viewed as  a GW tree
with immigration, see {\sc Athreya} and {\sc Ney} \cite{an:bp}.
We want to stress that we only consider here local limits i.e. we look
at the trees up to a fixed height $h$.
Other limits can be
considered such as scaling limits of conditioned GW trees
(see \cite{d:ltcpcgwt,k:ltcnggwt,r:slmbtgwtcnvodgs}) but this is not
the purpose here.

\esp It  is  also   known  that,  at  least  in   the  critical  case,  other
conditionings  such  as conditioning  by  the  total  progeny, see  {\sc
  Kennedy}   \cite{k:gwctp}  and  {\sc  Geiger}   and  {\sc  Kaufmann}
\cite{gk:slgwtiv},  or by the  number of  leaves, see {\sc Curien} and {\sc Kortchemski}
\cite{ck:rncpccgwta},
lead to the same local limit  in distribution. See also the survey from {\sc
  Janson} \cite{j:sgtcgwrac}.  

\esp For all those cases, the conditioning event can be written as
 $\{\tau\in \A_n\}$ with $\A_n$ of the form:
\[
\A_n=\{\bt,\ A(\bt)\ge n\}\quad\mbox{or}\quad \A_n=\{\bt,\ A(\bt)=n\},
\]
where $A:  \bt \mapsto  A(\bt)$ is  a functional defined  on the  set of
trees    and   satisfying   an    additive   property,    see   Equation
\reff{eq:A=A+B}.  The  main  result of this paper,  see  Theorem
\ref{theo:KGen} for a precise  statement, unifies all the previous
conditionings and  gives a necessary and sufficient  condition to obtain
Kesten's tree as  a limit.  In the non-degenerate  critical case, if $A$
 satisfies  the additive  property \reff{eq:A=A+B}
, then the following two
statements are equivalent (with some additional aperiodic condition for the converse):
\begin{itemize}
   \item $\lim_{n\rightarrow+\infty } \P(\tau\in \A_{n+1})/\P(\tau\in
  \A_{n})=1$,
   \item The distribution of $\tau$ conditionally on $\{\tau\in \A_n\}$
     converges to the distribution of Kesten's tree $\tau^*$. 
\end{itemize}

Using this result, we give elementary proofs for the convergence in
distribution to Kesten's tree $\tau^*$ of the GW tree conditioned on:
\begin{itemize}
\item[(i)]  Extinction after or at a large  time (sub-critical  and critical
 case), with $A(\bt)=H(\bt)$ the height of the tree $\bt$ and
 conditioning event $\{H(\tau)=n\}$ or $\{H(\tau)\geq n\}$. See Sections
 \ref{ex:H>n} and  \ref{ex:H=n}. 
\item[(ii)]   Large  total   population  size   (critical   case),  with
  $A(\bt)=\Card(\bt)$ the total size  of the tree and conditioning event
  $\{\Card(\tau)=n\}$   or   $\{\Card(\tau)\geq   n\}$.   See   Section
  \ref{sec:totI}.
\item[(iii)]  Large  number  of  leaves (critical  case),  with  $A(\bt)
  =L_0(\bt)$  the   total  number  of  leaves  of   $\bt$ and
 conditioning event $\{L_0(\tau)=n\}$ or $\{L_0(\tau)\geq n\}$.  See  Section
  \ref{sec:L0}.
\end{itemize}

\esp Let  us   mention  that  assertion  (i)  with   the  conditioning  event
$\{H(\tau)\geq n\}$ was first  proved by {\sc Kesten} \cite{k:sbrwrc} in
the   critical  case  under   a  finite   variance  condition,   and  in
\cite{j:sgtcgwrac}, Theorem  7.1, in  full generality. Property  (ii) is
also proved  in full generality in \cite{j:sgtcgwrac},  Theorem 7.1 (the
sub-critical  case  is  also  studied  in  \cite{j:sgtcgwrac},  see  the
discussion below). Finally, assertion  (iii) with the conditioning event
$\{L_0(\tau)=n\}$ has been proved  by {\sc Curien} and {\sc Kortchemski}
\cite{ck:rncpccgwta}, Theorem  4.1, in the critical  and finite variance
case only.

\esp In fact  the conditioning on the  large total population size  or on the
large number  of leaves  are particular cases  of conditioning  trees on
large  number of  individuals with  a  given number  of children.   This
corresponds to the functional  $A(t)=L_{\ca}(\bt)$ which gives the total
number of individuals of the tree $\bt$ whose number of children belongs
to a  given set  $\ca$ of nonnegative  integers.  Such  conditioning has
already been studied by {\sc Rizzolo} \cite{r:slmbtgwtcnvodgs}, see also
{\sc  Mimami \cite{m:nvgdgwt}}, but  for global  scaling limits  and not
local  limits.  We obtain  the convergence  in distribution  to Kesten's
tree  $\tau^*$ of  a  critical  GW tree  without  any additional  moment
condition on the offspring distribution, conditioned on:
\begin{itemize}
\item[(iv)] Large number of individuals with number of children in a given
  set $\ca$ (critical case),  with $A(\bt) =L_\ca(\bt)$ and conditioning
  event  $\{L_\ca(\tau)=n\}$  or  $\{L_\ca(\tau)\geq  n\}$. 
\end{itemize}
Here, we use the fact that $L_\ca(\tau)$ is distributed according to the
total progeny of another critical GW tree, which allows to use (ii), see
\cite{m:nvgdgwt,r:slmbtgwtcnvodgs}.   Let  us   remark  that  the  total
progeny $(\ca=\N)$, the number of leaves $(\ca=\{0\})$ and the number of
internal  nodes $(\ca=\N\setminus\{0\})$  are particular  cases  of this
conditioning.

\esp The main ingredients in the proof for (ii), (iii) and (iv) are Dwass
formula for the representation of the total progeny of a GW
tree using random walks, and the strong ration theorem for these
random walks which has some links with the local sub-exponential
property  of the total progeny of GW trees, see \cite{afk:asrvlsb}.

\esp We then  study the  subcritical case and  define a  one-parameter family
$(p_\theta,\theta\in I)$  of distributions on  the set of  integers such
that the GW  tree $\tau$ associated with the  offspring distribution $p$
and the GW tree $\tau_\theta$ associated with the offspring distribution
$p_\theta$ have  the same  conditional distributions given  $L_\ca$, see
Proposition  \ref{prop:sub}.  This generalizes  Kennedy's transformation
\cite{k:gwctp}  concerning the total  progeny, and  the pruning  of {\sc
  Abraham}, {\sc Delmas} and {\sc He} \cite{adh:pgwttvmp} concerning the
number of leaves.   According to \cite{j:sgtcgwrac}, we say  that $p$ is
generic (with  respect to  $\ca$) if there  exists $\theta_c$  such that
$p_{\theta_c}$ is  critical.  We then immediately  deduce, see Corollary
\ref{cor:K-sub}, that if $p$ is generic, then the distribution of $\tau$
conditionally  on  $\{L_\ca(\tau)=n\}$ (in  the  aperiodic  case) or  on
$\{L_\ca(\tau)\geq n\}$  converges to  the distribution of  the Kesten's
tree   $\tau^*_{\theta_c}$  associated   with  the   critical  offspring
distribution $p_{\theta_c}$.   When there is no such  $\theta_c$, then a
condensation phenomenon  may appear: {\sc Jonsson}  and {\sc Stefansson}
\cite{js:cnt} or  \cite{j:sgtcgwrac} proved for the  conditioning on the
total progeny that  the limiting tree in that case  is not Kesten's tree
but a tree with a unique  node with an infinite number of offsprings. We
shall investigate  this condensation phenomenon  for other conditionings
in a forthcoming paper \cite{ad:ncgwtcc}. Let us add that an
example is given in \cite{ad:ncgwtcc} of an offspring distribution
which is generic with respect to a set $\ca$ and non-generic with
respect to another set. Hence, it seems difficult to give a sufficient
condition for an offspring distribution to be generic (i.e. to have
existence of the critical value $\theta_c$).

\esp Finally, we  consider another conditioning which does  not enter in
the framework of  Theorem \ref{theo:KGen} : conditioning on  the size on
the  $n$-th generation.   However, we  can  adapt the  proof of  Theorem
\ref{theo:KGen} to get an analogous result in that case, see Proposition
\ref{prop:gen}.  We apply this  result to a critical geometric offspring
distribution where explicit computations  can be performed to prove that
the  corresponding GW tree  conditioned on  the $n$-th  generation being
positive but  smaller that $n^2$  converges in distribution  to Kesten's
tree. Using results on local limit of GW processes from {\sc Nagaev} and
{\sc Vakhtel} \cite{nv:ltpldgwp, nv:lltcwp}, this result can be extended
to very general critical offspring distributions.

\esp The paper  is organized as  follows.  In Section  \ref{sec:tree}, we
recall the  framework we use for  discrete trees and define  the GW tree
$\tau$  and Kesten's tree  $\tau^*$ associated  with offspring distribution
$p$.  In  Section \ref{sec:main}, we  state and prove the  necessary and
sufficient condition for convergence in distribution of the conditioned
tree to Kesten's  tree. We apply this result  in Section \ref{sec:ex} to
recover  the classical  results  on critical  conditioned  GW trees  and
we study  in Section  \ref{sec:LA} the  case of  the  number of
individuals  with  out-degree  in  a  given set  for  the  critical  and
sub-critical case. Finally, we study in Section \ref{sec:geom} the
conditioning on the size of the $n$-th generation of the GW tree.

\section{Technical background on GW trees}
\label{sec:tree}
\subsection{First notations}

We denote by $\N=\{0,1,2,\ldots\}$ the set of non-negative integers
 and by $\N^*=\{1,2,\ldots\}$ the set of positive integers. 

If $K$ is a subset of $\N^*$, we call the span of $K$ the greatest
common divisor of $K$.
If $X$ is an integer-valued random variable, we call the span of $X$
the span of  $\{n>0,\ \P(X=n)>0\}$ the restriction to $\N^*$ of its support.

\subsection{The set of discrete trees}
We recall Neveu's formalism \cite{n:apghw} for ordered rooted trees. We let
\[
\cu=\bigcup _{n\ge 0}{(\N^*)^n}
\]
be the set  of finite  sequences of positive  integers with  the convention
$(\N^*)^0=\{\emptyset\}$.  For  $u\in  \cu$  let  $|u|$  be  the  length  or
generation  of  $u$   defined  as  the  integer  $n$   such  that  $u\in
(\N^*)^n$. If $u$ and $v$ are  two sequences of $\cu$, we denote by $uv$
the concatenation of the two  sequences, with the convention that $uv=u$
if $v=\emptyset$ and  $uv=v$ if $u=\emptyset$.  The set  of ancestors of
$u$ is the set:
\begin{equation}
   \label{eq:Au}
   A_u=\{v\in \cu; \text{there exists $w\in \cu$, $w\neq \emptyset$,  such that $u=vw$}\}.
\end{equation}
The most recent common ancestor of a subset $\bs$ of $ \cu$, denoted by
$\text{M}(\bs)$, is the unique element $u$ of $\bigcap_{u\in \bs} A_u$ with
maximal length $|u|$.

For $u,v\in\cu$, we denote by $u<v$ the lexicographic order on $\cu$
i.e. $u<v$ if $u\in A_v$ or, if we set $w=\text{M}(u,v)$, then $u=wiu'$ and
$v=wjv'$ for some $i,j\in\N^*$ with $i<j$. 

A tree $\bt$ is a subset of $\cu$ that satisfies:
\begin{itemize}
\item $\emptyset\in\bt$,
\item If  $u\in\bt$, then $A_u\subset \bt$. 
\item For every $u\in \bt$, there exists a non-negative integer
  $k_u(\bt)$ such that, for every positive integer $i$,  $ui\in \bt$ iff $1\leq i\leq k_u(\bt)$. 
\end{itemize}

The integer $k_u(\bt)$ represents the number of offspring of the vertex
$u\in \bt$.  The vertex $u\in \bt$  is called a leaf  if $k_u(\bt)=0$. The
vertex $\emptyset$ is called the root of $\bt$.  Let us remark that, for
a tree $\bt$, we have
\begin{equation}\label{eq:sum_k}
\sum_{u\in\bt}k_u=\Card(\bt)-1.
\end{equation}

Let $\bt$ be a tree. The set of its leaves is  $\cl_0(\bt)=\{u\in \bt;
k_u(\bt)=0\}$, its height is defined by 
\[
H(\bt)=\sup\{|u|,\ u\in\bt\}
\]
and can be infinite. For $u\in \bt$, we define the 
sub-tree  $\cs_u(\bt)$ of $\bt$ ``above'' $u$ as:
\[
\cs_u(\bt)=\{v\in\cu,\ uv\in\bt\}.
\]

We   denote  by   $\T$  the   set  of   trees,  by 
$$\T_0=\{\bt\in  \T;\,\Card(\bt)<+\infty \}$$ 
the subset  of finite trees, by 
$$\T^{(h)}=\{\bt\in \T;  H(\bt)\leq h\}$$
 the  subset of trees  with height at  most $h\in\N$,   and  by   
$$\T_1=\{  \bt   \in  \T;   \lim_{n\rightarrow+\infty
 }|\text{M}(\{u\in \bt; |u|=n\})|=+\infty \}$$
  the subset of trees with
a unique infinite spine. Notice that $\T_0$ and $\T^{(h)}$ are countable
and $\T_1$ is uncountable as the set of infinite sequences of positive
integers can be embedded in $\T_1$.  For $h\in \N$ the restriction function $r_h$
from $\T$ to $\T$ is defined by:
\[
r_h(\bt)=\{u\in\bt,\ |u|\le h\}.
\]
We endow the set $\T$ with the ultrametric distance
\[
d(\bt,\bt')=2^{-\max\{h\in\N,\ r_h(\bt)=r_h(\bt')\}}.
\]
A  sequence $(\bt_n, n\in\N)$  of trees
converges to a tree $\bt$ with  respect to the distance $d$ if and only if,
for every $h\in \N$,
\[
r_h(\bt_n)=r_h(\bt)\qquad\mbox{for $n$ large enough}.
\]
The Borel $\sigma$-field associated with  the distance $d$ is the smallest
$\sigma$-field  containing  the singletons  for  which the  restrictions
functions  $(r_h, h\in  \N)$ are  measurable.  With  this  distance, the
restriction functions are contractant. Since $\T_0$ is dense in $\T$ and
$(\T,d)$  is complete,  we get that  $(\T,d)$ is  a  Polish metric
space.

Consider  the closed  ball  $B(\bt,2^{-h})=\{\bt'\in \T;
d(\bt,\bt')\leq 2^{-h}\}$ for some $\bt\in \T$ and $h\in \N$ and notice that:
\[
B(\bt, 2^{-h})=r_h^{-1}(\{r_h(\bt)\}).
\]
Since the  distance is  ultrametric, the closed  balls are open  and the
open balls are closed, and the intersection of two balls is either empty
or one of  them. We deduce that the  family $((r_h^{-1}(\{\bt\}), \bt\in
\T^{(h)}), h\in \N)$ is a  $\pi$-system, and Theorem 2.3 in \cite{b:cpm}
implies that this  family is convergence  determining for the  convergence in
distribution.   Let $(T_n,  n\in \N^*)$  and $T$  be  $\T$-valued random
variables. We denote by $\dist(T)$ the distribution of the random
variable $T$ (which is uniquely determined by the sequence of
distributions of $r_h(T)$ for every $h\ge 0$), and we denote
$$\dist(T_n)\underset{n\to+\infty}{\longrightarrow} \dist(T)$$
for the convergence in distribution of the sequence $(T_n,n\in\N^*)$
to $T$. 
We deduce from  the portmanteau  theorem that  the sequence
$(T_n, n\in  \N^*)$ converge in distribution  to $T$ if and  only if for
all $h\in \N$, $\bt\in \T^{(h)}$:
\[
\lim_{n\to+\infty}\P(r_h(T_n)=\bt)=\P(r_h(T)=\bt).
\]
For $\bt\in  \T$ and $u\not\in \bt$, set  $k_u(\bt)=-1$. The convergence
in distribution of the sequence $(T_n,n\in\N^*)$ to $T$ is also equivalent to
the  finite dimensional  convergences  in distribution  of the  sequence
$((k_{u_1}(T_n),  \ldots, k_{u_m}(T_n)),  n\in  \N^*)$ to  $(k_{u_1}(T),
\ldots, k_{u_m}(T))$ for all $m\in \N^* $ and $u_1, \ldots, u_m\in \cu$.

As we  shall only consider $\T_0$-valued random  variables that converge
in distribution  to a  $\T_1$-valued random variable,  we shall  give an
alternative  characterization of  convergence in  distribution that  holds for
this restriction.  To present this result, we  introduce some notations.
If $\bt,\bs\in\T$ and $x\in\cl_0(\bt)$ we denote by:
\[
\bt\circledast (\bs,x)=\{u\in \bt\}\cup\{xv,v\in\bs\}
\]
the tree obtained by grafting the tree $\bs$ on the leaf $x$ of
the tree $\bt$.
For every $\bt\in\T$ and every $x\in\cl_0(\bt)$, we shall consider the
set of trees obtained by grafting a tree on the leaf $x$ of $\bt$:
\[
\T(\bt,x)=\{\bt\circledast (\bs,x),\ \bs\in \T\}.
\]
It is easy to see that $\T(\bt,x)$ is closed. It is also  open, as for
all $\bs \in \T(\bt,x)$ we have that $B(\bs, 2^{-H(\bt)-1})\subset
\T(\bt, x)$. 

Moreover, notice that the set $\T_1$ is a Borel subset of the set $\T$.

\begin{lem}
   \label{lem:cv-determing}
Let $(T_n,  n\in \N^*)$  and $T$  be  $\T$-valued random
variables which belong a.s. to $\T_0\bigcup \T_1$.
The sequence $(T_n,  n\in \N^*)$ converges in distribution to $T$ if and
only if for every $\bt\in\T_0$ and every $x\in\cl_0(\bt)$, we have:
\begin{equation}
   \label{eq:cv-determing}
\lim_{n\to+\infty}\P(T_n\in \T(\bt,x))=\P(T\in \T(\bt,x))\quad\mbox{and}\quad
\lim_{n\to+\infty}\P(T_n=\bt)=\P(T=\bt).
\end{equation}
\end{lem}

\begin{proof}
The subclass
$\cf=\{\T(\bt,x),\ \bt\in\T_0,\ x\in\cl_0(\bt)\}\cup\{\{\bt\},\
\bt\in\T_0\}$ of the Borel sets on $T_0\bigcup \T_1$ 
forms a $\pi$-system since we have
\[
\T(\bt_1,x_1)\cap \T(\bt_2,x_2)=\begin{cases}
\T(\bt_1,x_1) & \mbox{if }\bt_1\in \T(\bt_2,x_2),\\
\T(\bt_2,x_2) & \mbox{if }\bt_2\in \T(\bt_1,x_1),\\
\{\bt_1\} & \mbox{if } \bt_1=\bt_2 \mbox{ and }x_1\ne x_2,\\
\emptyset & \mbox{in the other cases}.
\end{cases}
\]

For  every $h\in\N$  and  every $\bt\in\T^{(h)}$,  we  have that  $\bt'$
belongs  to  $r_h^{-1}(\{\bt\}) \bigcap  \T_1$  if  and  only if  $\bt'$
belongs  to some  $\T(\bs,x)$ where  $x$ is  a leaf  of $\bt$  such that
$|x|=h$ and $\bs$ belongs  to $r_h^{-1}(\{\bt\}) \bigcap \T_0$ such that
$x$ is also a leaf of  $\bs$.  Since $\T_0$ is countable, we deduce that
$\cf$   generates  the  Borel   $\sigma$-field  on   $\T_0\cup\T_1$.  In
particular $\cf$ is a separating class on $\T_0\bigcup \T_1$. 


Since $A\in \cf$ is   closed and  open as
well, according to Theorem  2.3  of  \cite{b:cpm}, 
to prove that the family  $\cf$ is a convergence determining class, it is
enough to check  that for all $\bt \in \T_0\bigcup  \T_1$ and $h\in \N$,
there exists $A \in \cf$ such that:
\begin{equation}
   \label{eq:thm2.3}
\bt\in A\subset B(\bt, 2^{-h}).
\end{equation}
If  $\bt\in \T_0$,  this is  clear as  $\{\bt\}=B(\bt, 2^{-h})$  for all
$h>H(\bt)$.   If  $\bt\in  \T_1$,   for  all  $\bs\in  \T_0$  and  $x\in
\cl_0(\bs)$  such that  $\bt \in  \T(\bs, x)$,  we have  $\bt\in \T(\bs,
x)\subset  B(\bt,  2^{-|x|})$.    Since we can find such a $\bs$ and $x$
such that $|x|$ is arbitrary large, we  deduce  that  \reff{eq:thm2.3}  is
satisfied. This
proves that the family  $\cf$ is a convergence determining class  on $\T_0\bigcup \T_1$.

Since, for $\bt\in \T_0$ and $x\in \cl_0(\bt)$ the sets  $\T(\bt,x)$ and
$\{\bt\}$ are open and closed, we deduce from the portmanteau Theorem
that if $(T_n, n\in \N^*)$ converges in distribution to $T$, then
\reff{eq:cv-determing} 
holds for every $\bt\in\T_0$ and every $x\in\cl_0(\bt)$. 
\end{proof}



\subsection{GW trees}

Let $p=(p(n), n\in \N)$ be a  probability distribution on the set of the
non-negative integers. We assume that
\begin{equation}
\label{eq:cond-p}
p(0)>0,\ p(0)+p(1)<1,\ \mbox{and}\ \mu:=\sum_{n=0}^{+\infty}np(n)<+\infty.
\end{equation}

A $\T$-valued random variable $\tau$ is a Galton-Watson (GW) tree
with   offspring distribution     $p$   if    the    distribution   of
$k_\emptyset(\tau)$  is $p$  and  for $n\in  \N^*$, conditionally  on 
$\{k_\emptyset(\tau)=n\}$,                 the                 sub-trees
$(\cs_1(\tau),\cs_2(\tau),\ldots,\cs_n(\tau))$   are   independent   and
distributed as the original tree $\tau$.
Equivalently, for every  $h\in \N^*$ and every $\bt\in
\T^{(h)}$, we have
\[
\P(r_h(\tau)=\bt)=\prod_{u\in r_{h-1}(\bt)}p(k_u(\bt)).
\]
In particular, the restriction of the distribution of $\tau$ on the
set $\T_0$ is given by: 
\begin{equation}\label{eq:loi-tau}
\forall \bt\in \T_0,\quad \P(\tau=\bt)=\prod_{u\in\bt}p(k_u(\bt)).
\end{equation}
The GW tree is called critical (resp. sub-critical, super-critical) if
$\mu=1$ (resp. $\mu<1$, $\mu>1$). 
\subsection{Conditioning on non-extinction}
\label{sec:kesten-main}
Let   $p$   be   an offspring distribution    satisfying   Assumption
\reff{eq:cond-p}  with  $\mu\le  1$  (i.e.  the  associated  GW
process is critical or sub-critical). We denote by $p^*=(p^*(n)=np(n)/\mu,
n\in \N)$ the corresponding size-biased distribution.

We define an infinite random
tree $\tau^*$ (the size-biased tree that we call Kesten's tree in this
paper), whose distribution is as follows. There 
exists a unique infinite sequence  $(V_k, k\in \N^*)$ of positive
  integers such that, for every $h\in \N$, $V_1\cdots V_h\in \tau ^*$, with
  the convention that $V_1\cdots V_h=\emptyset$ if $h=0$.
The joint distribution of $(V_k, k\in \N^*)$ and $\tau^*$ is
  determined recursively as follows: for each $h\in \N$,
  conditionally given $(V_1,\ldots,V_h)$ and $r_h(\tau^*)$, we have:
\begin{itemize}
\item The number of children $(k_v(\tau^*),\ v\in \tau^*,\ |v|=h)$ are
  independent and  distributed according to $p$ if $v\ne
  V_1\cdots V_h$ and according to $p^*$ if $v=V_1\cdots V_h$.
\item  Given  also  the  numbers  of children  $(k_v(\tau^*),\  v\in  \tau^*,\
  |v|=h)$, the integer  $V_{h+1}$ is uniformly distributed on  the set of
  integers $\left\{1,\ldots, k_{V_1\cdots V_h}(\tau^*)\right\}$.
\end{itemize}

Notice that by construction, $\tau^*\in\T_1$ a.s.

Following Kesten \cite{k:sbrwrc}, the random
tree $\tau^*$ can be viewed as the tree $\tau$ conditioned on
non-extinction as:
\[
\forall h\in\N^*,\ \forall
\bt\in\T^{(h)},\ \P(r_h(\tau^*)=\bt)=\lim_{n\to+\infty}\P(r_h(\tau)=\bt\bigm|
H(\tau)\ge n).
\]

As a direct consequence we get that for all $h\in \N$, $\bt\in
\T^{(h)}$, $u\in \bt$ such that $|u|=h$:
\[
\P(r_h(\tau^*)=\bt, V_1\cdots V_h=u)=\mu^{-h} \P(r_h(\tau)=\bt),
\]
and for all $\bt\in \T_0$, $x\in \cl_0(\bt)$:
\begin{equation}
   \label{eq:t*-T}
\P(\tau^*\in \T(\bt,x))=\mu^{-|x|} \P(\tau\in \T(\bt,x)). 
\end{equation}
Since,  for  $\bt\in \T_0$ and  $x\in \cl_0(\bt)$, 
$ \P(\tau=\bt)=\P(\tau\in \T(\bt,x), k_x(\tau)=0) =
\P(\tau\in \T(\bt,x))p(0)$,
we deduce that:
\begin{equation}
   \label{eq:t*-Tp}
\P(\tau^*\in \T(\bt,x))=\inv{\mu^{|x|}p(0)}  \P(\tau=\bt).
\end{equation}
Since $\tau^*$ is in $\T_1$ a.s., this implies that \reff{eq:t*-Tp} with
$\bt\in \T_0$ and  $x\in \cl_0(\bt)$ characterizes the distribution of
$\tau^*$.

\section{Main result}\label{sec:main}

Let $A$ be an integer-valued function defined on $\T$ which is finite on
$\T_0$ and satisfies  the
following additivity property:  there exists an
integer-valued 
function $D$ defined on $\T$ such that, for every $\bt\in\T_0$,
every $x\in\cl_0(\bt)$ and for every $\tilde\bt$ such that
$A(\bt\circledast(\tilde\bt,x)) $ is large enough,
\begin{equation}
   \label{eq:A=A+B}
A(\bt\circledast(\tilde\bt,x))=A( \tilde \bt)+D(\bt,x). 
\end{equation}

Let $n_0\in \N\cup\{+\infty  \}$ be given. We define  for all $n\in \N^*$,
the subset of trees
$$\A_n=\{\bt\in \T; A(\bt)\in [n, n+n_0)\}.$$  
Common values of $n_0$ that will be considered are $1$ and $+\infty$. 

The   following  theorem   states   that  the   distribution  of   the
GW tree $\tau$ conditioned to be in $\A_\infty$, the limit of
$\A_n$, is distributed as $\tau^*$ as soon as the probability of $\A_n$
satisfies some regularity.
We denote by
$$\dist(\tau|\tau\in\A_n)$$
the conditional law of $\tau$ given $\{\tau\in\A_n\}$.

\begin{theo}
   \label{theo:KGen}
   Assume that Assumptions  \reff{eq:cond-p} and \reff{eq:A=A+B} hold,
   that $\P(\tau\in \A_n)>0$ for $n$ large enough and that  one of the two following conditions
\begin{itemize}
\item $\mu=1$ or
\item $\mu<1$ and $D(\bt,x)=|x|$ for all $\bt\in \T_0$, $x\in
  \cl_0(\bt)$. 
\end{itemize}
Then, if 
\begin{equation}
   \label{eq:CondK}
\lim_{n\rightarrow+\infty } \frac{\P(\tau\in \A_{n+1})}{\P(\tau\in
  \A_{n})}=\mu,
\end{equation}
we have:
$$\dist(\tau|\tau\in\A_n)\underset{n\to+\infty}{\longrightarrow} \dist(\tau^*).$$
Conversely,  if   $\dist(\tau|\tau\in\A_n)\underset{n\to+\infty}{\longrightarrow} \dist(\tau^*)$  and  if the span
of $\{D(\bt,x); \bt\in \T_0\text{  and } x\in \cl_0(\bt)\}\bigcap \N^* $
is one, then \reff{eq:CondK} holds.
\end{theo}

Recall that the local convergence in distribution towards $\tau^*$ is
equivalent to 
\begin{equation}
   \label{eq:limt-An}
\forall h\in\N^*,\ \forall
\bt\in\T^{(h)},\ \lim_{n\to+\infty}\P(r_h(\tau)=\bt\bigm|\tau\in \A_n
)=\P(r_h(\tau^*)=\bt).
\end{equation}

\begin{proof}
Let us first remark that, as we supposed that $\mu\le 1$, we have a.s.
$\tau\in\T_0$ and thus we are in the setting of Lemma \ref{lem:cv-determing}.

 Using  \reff{eq:loi-tau}, we have for every $\bt\in \T_0$, $x\in\cl_0(\bt)$ and  $\tilde\bt\in \T_0$:
\[
\P(\tau=\bt\circledast(\tilde\bt,x))=\inv{p(0)} \P(\tau=\bt)\P(\tau=\tilde
\bt).
\]
Let $\bt\in \T_0$ and $x\in \cl_0(\bt)$. Then, if $n$ is large enough so that we can apply Equation
\reff{eq:A=A+B}, we get:  
\begin{align*}
\P(\tau\in \T(\bt,x) , \tau\in \A_n)
&=\sum_{\tilde\bt\in \T_0} \P(\tau=\bt\circledast(\tilde \bt,x))\ind_{\{
  n\leq A(\bt\circledast(\tilde\bt,x)) <n+n_0\}}\\
&=\inv{p(0)}\sum_{\tilde\bt\in \T_0} \P(\tau=\bt)\P(\tau=\tilde \bt)\ind_{\{ n\leq A(\tilde \bt)+ D(\bt,x) <n+n_0\}}\\
&=\inv{p(0)} \P(\tau=\bt) \P(n-D(\bt,x)\leq A(\tau)< n+n_0-D(\bt,x))\\
&=\mu^{|x|} \P(\tau^*\in \T(\bt,x)) \P(\tau\in \A_{n-D(\bt,x)}), 
\end{align*}
where we used  \reff{eq:t*-Tp}
for the last equality. Therefore we have
\begin{equation}
   \label{eq:tt*}
\P(\tau\in \T(\bt,x) \bigm|\tau\in\A_n)
= \P(\tau^*\in \T(\bt,x)) \, \mu^{|x|}\, \frac{\P(\tau\in
  \A_{n-D(\bt,x)})}{
  \P(\tau\in \A_n)}\cdot
\end{equation}
Then, using \reff{eq:CondK} and that $D(\bt,x)=|x|$ if
$\mu<1$, we obtain that:
\begin{equation}
   \label{eq:tt*2}
\lim_{n\rightarrow+\infty } \P(\tau\in \T(\bt,x) \bigm| \tau\in\A_n)
=\P(\tau^*\in \T(\bt,x)).
\end{equation}
For all $\bt\in \T_0$ and all $n>A(\bt)$, we have
$$\P(\tau=\bt,
  \tau\in \A_n)=\P(\tau=\bt,
  \bt\in \A_n)\le \ind_{\{\bt\in\A_n\}}=0$$
and thus:
\begin{equation}
   \label{eq:tt*3}
\lim_{n\rightarrow+\infty } \P(\tau=\bt \bigm| \tau\in\A_n)=0
=\P(\tau^*=\bt).
\end{equation}
We deduce from Lemma \ref{lem:cv-determing} that \reff{eq:limt-An} holds. 

Conversely, if \reff{eq:limt-An} holds, then Lemma \ref{lem:cv-determing}
implies that \reff{eq:tt*2} and \reff{eq:tt*3} hold. 
The fact
that 
the span of $\{D(\bt,x); \bt\in \T_0\text{ and } x\in
\cl_0(\bt)\}\bigcap \N^* $ is one and  \reff{eq:tt*} 
imply, with Bezout theorem,  that \reff{eq:CondK} holds. 
\end{proof}

\section{Examples}
\label{sec:ex}


\subsection{Conditioning on extinction after large time}
\label{ex:H>n}

We give here a simple proof of Kesten's result for the convergence in
distribution of a critical or sub-critical GW tree conditioned on
non-extinction, see \cite{k:sbrwrc} under a finite variance condition
and \cite{j:sgtcgwrac} for the general case.

\begin{prop}
  Let $\tau$ be a critical or sub-critical GW tree with offspring distribution
  $p$ satisfying Assumption \reff{eq:cond-p}. Then, we have
\begin{equation}\label{eq:H>n_2}
\dist(\tau|H(\tau)\ge n)\underset{n\to+\infty}\longrightarrow \dist(\tau^*).
\end{equation}
\end{prop}

\begin{proof}
 Consider $A(\bt)=H(\bt)$ and
     $n_0=+\infty $  that is $\A_n=\{\bt\in\T; \, H(\bt)\geq
     n\}$. 
Notice that in this case for a tree $\tilde\bt$ such that $H(\tilde
     \bt)$ is larger than $H(\bt)$, we have for every $x\in\cl_0(\bt)$
\begin{equation}
   \label{eq:AD-height}
A(\bt\circledast(\tilde\bt,x))=A(\tilde \bt)+|x|. 
\end{equation}
Therefore, Condition \reff{eq:A=A+B} is satisfied by $A$.  

According to Theorem \ref{theo:KGen}, it suffices to prove
\begin{equation}\label{eq:H>n_1}
\lim_{n\to+\infty}\frac{\P(H(\tau)\ge n+1)}{\P(H(\tau)\ge n)}=\mu
\end{equation}
to get \reff{eq:H>n_2}.

We denote by $\varphi$ the
generating function of $p$ and we define recursively $\varphi_1=\varphi$
and for $n\ge  1$, $\varphi_{n+1}=\varphi_n\circ\varphi$. As $\varphi_n$
is  the  generating  function  of  the distribution  of  $\{u\in \tau;
|u|=n\}$ the  number  of
individuals  at height $n$,  we have  $\P(\tau\in \A_n)=1-\varphi_n(0)$.
We also have $\lim_{n\rightarrow+\infty } \varphi_n(0)=1$ and
\[
\lim_{n\rightarrow+\infty } \frac{\P(\tau\in 
     \A_{n+1})}{\P(\tau\in 
     \A_{n})}
=
\lim_{n\rightarrow+\infty } \frac{1-\varphi(\varphi_{n}(0))}{1- \varphi_n(0)}
= \varphi'(1)=\mu
\]
which is \reff{eq:H>n_1}.
\end{proof}

\subsection{Conditioning on extinction at large time}
\label{ex:H=n}

\begin{prop}
  Let $\tau$  be a  critical or sub-critical  GW tree
  with offspring distribution $p$ satisfying Assumption \reff{eq:cond-p}. Then
  we have
\begin{equation}\label{eq:H=n_2}
\dist(\tau|H(\tau)=n)\longrightarrow \dist(\tau^*).
\end{equation}
\end{prop}

\begin{proof}
We consider $A(t)=H(\bt)$  with 
  $n_0=1$   that  is   $\A_n=\{\bt\in\T; \;  H(\bt)=   n\}$. 
Since \reff{eq:AD-height} is in force, we get that
Condition \reff{eq:A=A+B} still holds. Again it suffices to prove
\begin{equation}\label{eq:H=n_1}
\lim_{n\to+\infty}\frac{\P(H(\tau)= n+1)}{\P(H(\tau)= n)}=\mu
\end{equation}
to get \reff{eq:H=n_2}.
Recall notation
  $\varphi_n$ introduced in Section \ref{ex:H>n} and that
  $\lim_{n\rightarrow+\infty } \varphi_n(0)=1$. 
We   have   $\P(\tau\in
  \A_n)=\varphi_{n+1}(0)- \varphi_{n}(0)$ and:
\[
\lim_{n\rightarrow+\infty } \frac{\P(\tau\in 
     \A_{n+1})}{\P(\tau\in 
     \A_{n})}
=
\lim_{n\rightarrow+\infty } \frac{
\frac{1-\varphi(\varphi_{n}(0))}{1- \varphi_n(0)}
- \frac{1-\varphi_2(\varphi_{n}(0))}{1- \varphi_n(0)}
}{
1
-\frac{1-\varphi(\varphi_{n}(0))}{1- \varphi_n(0)}
}
= \frac{\mu-\mu^2}{1-\mu}=\mu,
\]
which is \reff{eq:H=n_1}.
\end{proof}

\subsection{Conditioning on the total population size, critical case}
\label{sec:totI}

We recover here  results from  
Theorem 7.1 in \cite{j:sgtcgwrac} 
on the convergence in distribution
of a critical GW tree conditioned on the size of its total
progeny to Kesten's tree.

Our proof is based on
Dwass formula (see \cite{d:tpbprrw}) that we recall now.
Let $(\tau_k, k\in \N^*)$ be independent GW trees distributed as
$\tau$. Set $W_k=\Card(\tau_k)$. Let $(X_k, k\in \N^*)$ be independent
integer-valued random variables distributed according to $p$.
For $k\in \N^*$ and $n\geq k$,
we have:
\begin{equation}
   \label{eq:Dwass}
\P(W_1+\ldots+W_k=n)=\frac{k}{n}\P(X_1+ \ldots+X_n=n-k).
\end{equation}

We also recall some results on random walks. 
Let $Y$  be an  integrable random variable  taking values in  $\Z$, such
that $\E[Y]=0$, $\P(Y=0)<1$ and the span  of $|Y|$ is 1. We consider the
random walk $S=(S_n, n\in \N)$ defined by:
\begin{equation}
   \label{eq:defS}
S_0=0 \quad\text{and}\quad S_n=\sum_{k=1}^n Y_k \quad\text{for $n\in \N^*$}.
\end{equation}
 Then the 
random walk $S$ is  recurrent. We define the period of $S$ as the span
of the set $\{n>0,\ \P(S_n=0)>0\}$. If $S$  is aperiodic (i.e. has
period 1), the
strong ratio  theorem for recurrent aperiodic random  walks, see Theorem
T1 p49 of \cite{s:prw}, gives that, for $\ell \in \Z$:
\begin{equation}
   \label{eq:SRTh}
\lim_{n\rightarrow+\infty } \frac{\P(S_n=\ell)}{\P(S_n=0)}=
\lim_{n\rightarrow+\infty} \frac{\P(S_n=0)}{\P(S_{n+1}=0)}=1.
\end{equation}

If $S$  has period $d$, then for all $k\in \{1, \ldots, d\}$, there
exist $j_k \in \Z$ and $n_k\in\N^*$
such that 
\begin{equation}
   \label{eq:jd}
\forall n\ge n_k,\quad\P(S_{nd+k}=j_k)>0.
\end{equation}
The strong ratio  theorem can then easily be adapted 
to get that, for $\ell \in
\Z$, $k\in \{ 1, \ldots, d\}$:
\begin{equation}
   \label{eq:SRlem}
\lim_{m\rightarrow+\infty } \frac{\P(S_{md+k}=\ell d+j_k)}{\P(S_{md}=0)}=1.
\end{equation}
Notice that \reff{eq:Dwass} and \reff{eq:SRlem} directly imply that
the total progeny distribution enjoys the local sub-exponential
property, see \cite{afk:asrvlsb}.

\begin{prop}\label{prop:totI}
Let  $\tau$ be  a critical GW  tree with  offspring distribution   $p$ satisfying
  Assumption   \reff{eq:cond-p}.  Let $d$ be the span of
  $\Card(\tau)-1$ (that is the span of the set
$\{k>0,\ p(k)>0\}$). Then we have
\begin{equation}\label{eq:totI_2}
\dist(\tau|\Card(\tau)=nd+1)\underset{n\to+\infty}{\longrightarrow} \dist(\tau^*)
\end{equation}
and
\begin{equation}\label{eq:totII_2}
\dist(\tau|\Card(\tau)\ge n)\underset{n\to+\infty}{\longrightarrow} \dist(\tau^*).
\end{equation}
\end{prop}

\begin{rem}
\label{rem:recip}
If we consider $A(\bt)=\Card (\bt)$ and $n_0=+\infty$ that is 
$\A_n=\{\bt\in\T,\ \Card(\bt)\ge n\}$, the converse of
Theorem \ref{theo:KGen} gives the sub-exponential property:
\begin{equation}\label{eq:totII_1}
\lim_{n\to+\infty}\frac{\P(\Card(\tau)\ge n+1)}{\P(\Card(\tau)\ge n)}=1.
\end{equation}
\end{rem}

\begin{proof}[Proof of Proposition \ref{prop:totI}]
Consider $A(\bt)=\Card(\bt)$  and
  $n_0=d$. Then we have
\[
\A_n=\{\bt\in\T;\;   \Card(\bt)\in [n, n+d)\}.
\]
We have for every $\bt\in\T$, without any additional assumption,
\begin{equation}
   \label{eq:A=2A}
A(\bt\circledast(\tilde\bt,x))=A(\tilde \bt)+A(\bt),
\end{equation}
so Condition \reff{eq:A=A+B} holds.
Again, it therefore suffices to prove
\begin{equation}\label{eq:totI_1}
\lim_{n\to+\infty}\frac{\P(\Card(\tau)\in[n+1,n+1+d))}{\P(\Card(\tau)\in[n,n+d))}=1
\end{equation}
to get \reff{eq:totI_2}.
By the definition of $d$, a.s. we have $A(\tau)\in d\N+1$. We consider
an integer valued random variable $X$ distributed according to $p$ and
we set $Y=X-1$ so that $\E[Y]=0$ since we supposed that $\mu=1$.
 The random walk defined by \reff{eq:defS}
has period $d$ and we can choose $j_1=-1$ in \reff{eq:jd} as
$\P(Y=-1)>0$.  Dwass formula
\reff{eq:Dwass} implies that, for $k=\lfloor (n-1)/d \rfloor$:
\[
\P(\tau\in \A_n)=\P(A(\tau)\in [n,n+d))=\P(A(\tau)=kd+1)=
\inv{kd+1}\P(S_{kd+1}=-1). 
\]
Using \reff{eq:SRlem},  we deduce that:
\[
\lim_{n\rightarrow+\infty} \frac{\P(\tau\in \A_{n+1})}{\P(\tau\in \A_n)}
= \lim_{k\rightarrow+\infty} \frac{\P(S_{(k+1)d+1}=-1)}{\P(S_{kd+1}=-1)}=1
\]
which readily implies \reff{eq:totI_1}.

The second assertion \reff{eq:totII_2} is then a straightforward
consequence of \reff{eq:totI_1}.
\end{proof}

\begin{rem}
   \label{rem:equiv}
Notice that the local limit theorem  gives asymptotics for $\P(S_n=-1)$
when the distribution of $X$ belongs to the domain of attraction of a
stable law, see Theorem 4.2.1 of \cite{il:issrv} or Theorem 1.10 in
\cite{k:ipgwcnl}.  This gives asymptotics for $\P(\tau\in \A_n)$ which
in turns allow to recover  Condition  \reff{eq:CondK}. 
\end{rem}

\subsection{Conditioning on the number of leaves, critical case}
\label{sec:L0}

For a finite tree $\bt\in\T_0$, we denote by $L_0(\bt)=\Card
(\cl_0(\bt))$ the number of leaves of $\bt$. The next proposition
(which seems to be a new result) is in fact a particular case of the
proposition of the next section. However, we prove it separately for
methodological purpose as its proof and in particular the construction of
the GW tree that codes $\cl_0(\bt)$ of Remark
\ref{rem:GW-leaves} are much simpler in that particular case.

\begin{prop}
  Let  $\tau$ be  a critical  GW  tree with  offspring distribution   $p$ satisfying
  Assumption   \reff{eq:cond-p}. Let $d_0$ be the span of the random variable
$L_0(\tau)-1$. Then we have
\begin{equation}\label{eq:L0_2}
\dist(\tau|L_0(\tau)=nd_0+1)\underset{n\to+\infty}\longrightarrow \dist(\tau^*)
\end{equation}
and
\begin{equation}\label{eq:L0II_2}
\dist(\tau|L_0(\tau)\ge n)\underset{n\to+\infty}\longrightarrow \dist(\tau^*).
\end{equation}
\end{prop}

\begin{proof}
We consider $A(\bt)=L_0(\bt)$  and
  $n_0= d_0$   which yields $\A_n=\{\bt\in \T;\;   L_0(\bt) \in [n,
  n+d_0)\}$. 
We have for every trees $\bt,\tilde\bt\in\T_0$ and every $x\in\cl_0(\bt)$
\begin{equation}
   \label{eq:A=2A-1}
A(\bt\circledast(\tilde\bt,x))=A(\tilde \bt)+A(\bt)-1.
\end{equation}

 According  to \cite{m:nvgdgwt}, see also Remark
\ref{rem:GW-leaves} below, $L_0(\tau)$ is  distributed as the total size
of a  critical GW tree $\tau_0$ with  offspring distribution given
by the distribution of:
\begin{equation}
   \label{eq:reprod-leaf}
X_0=\sum_{k=1}^{N-1} Z_k,
\end{equation}
with $(Z_k, k\in \N^*)$ and  $N$  independent random variables such that 
$(Z_k, k\in \N^*)$ are  independent and  distributed as
$X-1$ conditionally on $\{X \geq 1\}$ (where $X$ is a random
variable distributed according to $p$) and $N$ has a 
geometric distribution with parameter $p(0)$. As $\E[X_0]=1$, we get that $\tau_0$ is
critical. Notice that $d_0$ is also the span of the random variable
$X_0$. 

It follows from \reff{eq:totI_1} that:
\begin{equation}\label{eq:L0_1}
\lim_{n\to+\infty}\frac{\P(L_0(\tau)\in[n+1,n+1+d_0))}{\P(L_0(\tau)\in[n,n+d_0))}=1.
\end{equation}
Then use  Theorem
\ref{theo:KGen} to get that \reff{eq:L0_2} holds.

If we consider $n_0=+\infty $ that is:
\[
\A_n=\{\bt\in\T_0;\;   L_0(\bt)\geq    n\},
\]
arguing as in the proof of the second part of Proposition \ref{prop:totI},
we get \reff{eq:L0II_2}.
\end{proof}

\begin{rem}
We deduce from Remark \ref{rem:recip} that  \reff{eq:L0II_2} implies
$$\lim_{n\to+\infty}\frac{\P(L_0(\tau)\ge n+1)}{\P(L_0(\tau)\ge n)}=1.$$
\end{rem}

\begin{rem}
   \label{rem:GW-leaves}
We shall briefly recall how one can prove that $L_0(\tau)$ is
distributed as the total size of a GW process by mapping the
set of leaves $\cl_0(\tau)$ onto a GW tree, see
\cite{m:nvgdgwt,r:slmbtgwtcnvodgs} for details. 

Let $\bt$ be a tree. 
For $u\in \bt$, we define the left branch starting from $u$ as:
\[
B_g^\bt(u)=\{uv;\, |v|\geq 1 \quad \text{and} \quad v=\{1\}^{|v|}\}\cap \bt.
\]
We also define the left leaf $G(u)$ of $u$ and the left ancestors $A_g(v)$ of a leaf $v$ as:
\[
G^\bt(u)=B_g^\bt(u)\cap \cl_0(\bt)
\quad\text{and}\quad
A_g^\bt(v)=\{u\in A_v; \, G^\bt(u)=v\}.
\]
For a leaf $v\in \cl_0(\bt)$, we define its leaf-children as:
\[
C^\bt(v)=\{G^\bt(ui); \, u\in A_g^\bt(v), 1<i\leq k_u(\bt)\},
\]
labeled  according to  the following  order:  $G^\bt(ui)<G^\bt(u'i')$ if
$u<u'$ in the lexicographic order  or if $u=u'$ and $i<i'$. This defines
a   tree,   obtained   from    the   leaves   of   $\bt$,   denoted   by
$\bt_{\{0\}}=F_{\{0\}}(\bt)$.       And       we       have       $\Card
(\bt_{\{0\}})=L_0(\bt)$.

If $\tau$ is a GW  tree then $\tau_{\{0\}}=F_{\{0\}}(\tau)$ is also a GW
tree  with  offspring distribution  given  by  the  distribution of  $X_0$  in
\reff{eq:reprod-leaf}.
\end{rem}

\begin{center}
\begin{figure}[H]\label{fig:L0}
\includegraphics[width=13cm]{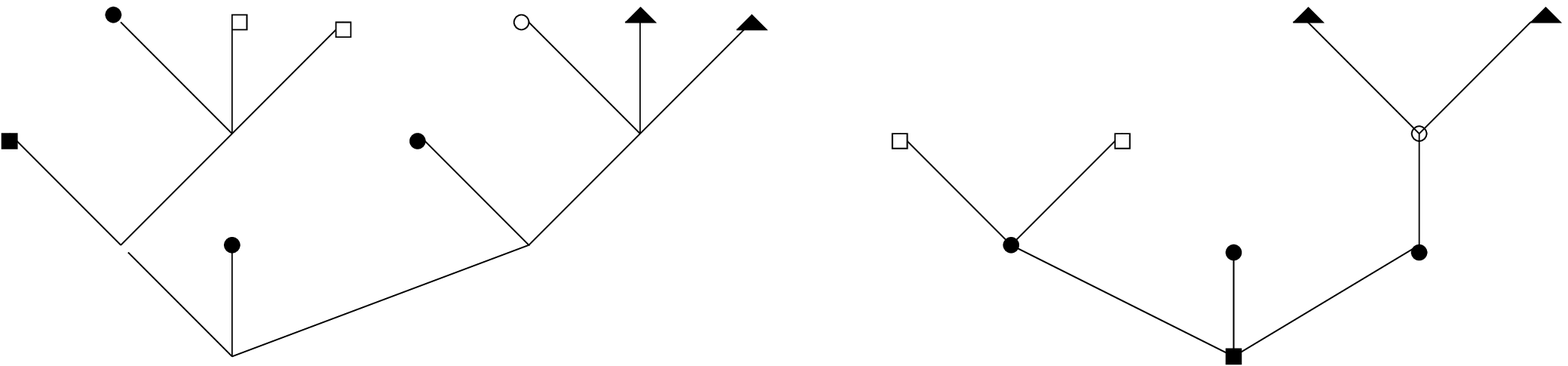}
\caption{A tree $\bt$ on the left and the coding of $\cl_0(\bt)$ by a
  tree $\bt_0=F(\bt)$ tree on the right.}
\end{figure}
\end{center}


\section{Conditioning on the number of individuals having a given
  number of children}
\label{sec:LA}

Let $\ca$ be a non-empty subset of $\N$. For a tree $\bt\in\T$, we write $\cl_\ca(\bt)=\{u\in \bt;
\; k_u(\bt)\in \ca\}$ the set of individuals whose number of children belongs
to $\ca$ and $L_\ca(\bt)=\Card(\cl_\ca(\bt))$ its cardinal. The case
$\ca=\{0\}$ represents the set of leaves of $\bt$ and has been treated
in Section \ref{sec:L0}. We can also have $L_\ca(\bt)=\Card(\bt)$ by
taking $\ca=\N$ or $L_\ca(\bt)$ can also be the number of internal nodes by
taking $\ca=\N^*$.

We set:
\[
p(\ca)=\sum_{k\in \ca}p(k).
\]

\subsection{The critical case}

Let us first remark that for every $\bt\in\T_0$, every $x\in\cl_0(\bt)$
and every $\tilde\bt\in\T$
\[
L_\ca(\bt\circledast(\tilde\bt,x))=\begin{cases}
L_\ca(\bt)+L_\ca(\tilde\bt)-1 & \mbox{if }0\in\ca,\\
L_\ca(\bt)+L_\ca(\tilde\bt) & \mbox{if }0\not\in\ca,
\end{cases}
\]
and hence $L_\ca$ satisfies the additive property \reff{eq:A=A+B} with
$D(\bt,x)=L_\ca(\bt)-\ind_{\{0\in \ca\}}$.

\begin{theo}\label{theo:LA}
Let  $\tau$ be  a critical GW  tree with offspring distribution   $p$ satisfying
  Assumption   \reff{eq:cond-p} and such that $p(\ca)>0$.  Let $d_\ca$ be the span of the random variable
$L_\ca(\tau)-1$. Then we have
\begin{equation}\label{eq:LA_2}
\dist(\tau|L_\ca(\tau)=nd_\ca+1)\underset{n\to+\infty}{\longrightarrow} \dist(\tau^*)
\end{equation}
and
\begin{equation}\label{eq:LA2_2}
\dist(\tau|L_\ca(\tau)\ge n)\underset{n\to+\infty}{\longrightarrow} \dist(\tau^*).
\end{equation}
\end{theo}

\begin{rem}
It is interesting to note that previous works
\cite{r:slmbtgwtcnvodgs,k:ipgwcnl} studying conditioned GW 
trees involving $L_\ca$ required additional assumptions on the moments
of $p$ or on $\ca$ (finite variance offspring distribution and
$0\in\ca$ in \cite{r:slmbtgwtcnvodgs}, and offspring distribution $p$
in the domain of attraction of a stable law with either $\ca$ or
$\N\setminus\ca$ finite in the case of infinite variance offspring
distribution in \cite{k:ipgwcnl}).
\end{rem}

\begin{rem}
In the proof of Theorem \ref{theo:LA}, we will see that if $0\not\in\ca$, then $d_\ca=1$.
\end{rem}

\begin{rem}
As a corollary, we get the following result, which is proven using the
same technique as in Remark \ref{rem:recip}:
\begin{equation}\label{eq:LA2_1}
\lim_{n\to+\infty}\frac{\P(L_\ca(\tau)\ge n+1)}{\P(L_\ca(\tau)\ge n)}=1.
\end{equation}
\end{rem}

\begin{proof}[Proof of Theorem \ref{theo:LA}]

In what follows, we denote by $X$ a random variable distributed
according to $p$. 
We consider only $\P(X\in \ca)<1$, as the case $\P(X\in
\ca)=1$  corresponds  to the  critical  case  with  $\ca=\N$ of  Section
\ref{sec:totI}.    

For a tree $\bt$ such that
$\cl_\ca(\bt)\ne\emptyset$,  following \cite{r:slmbtgwtcnvodgs}, we
can map the set $\cl_\ca(\bt)$ onto a tree $\bt_\ca$. We first define
a map $\phi$ from $\cl_\ca(\bt)$ on $\cu$ and a sequence
$(\bt_k)_{1\le k\le n}$ of
trees (where $n=L_\ca(\bt)$) as follows.
Recall that we denote by $<$ the lexicographic order on $\cu$. Let
$u^1<\cdots<u^n$ be the ordered elements of $\cl_\ca(\bt)$.
\begin{itemize}
\item $\phi(u^1)=\emptyset$, $\bt_1=\{\emptyset\}$.
\item For $1<k\le n$, recall that $S_{M(\{u^{k-1},u^k\})}(\bt)$
  denotes the
  tree above the most recent common ancestor of $u^{k-1}$ and $u^k$,
  and we set $\bs=\{M(\{u^{k-1},u^k\})u,u\in
  S_{M(\{u^{k-1},u^k\})}(\bt)\}$ and $v=\min(\cl_\ca(\bs))$.
Then, we set
$$\phi(u^k)=\phi(v)(k_\phi(v)(\bt_{k-1})+1)$$
the concatenation of the node $\phi(v)$ with the integer $k_\phi(v)(\bt_{k-1})+1$,
and
$$\bt_k=\bt_{k-1}\cup\{\phi(u^k)\}.$$
In other words, $\phi(u^k)$ is a child of $\phi(v)$ in $\bt_k$ and we add it
``on the right'' of the other children (if any) of $\phi(v)$ in the previous
tree $\bt_{k-1}$ to get $\bt_k$.
\end{itemize}
It is clear by construction that $\bt_k$ is a tree for every $k\le n$. We
set $\bt_\ca=\bt_n$. Then $\phi$ is a one-to-one map from
$\cl_\ca(\bt)$ onto $\bt_\ca$.
The construction of the tree $\bt_\ca$ is illustrated on Figure \ref{fig:tA}.

\begin{figure}[H]
\includegraphics[width=12cm]{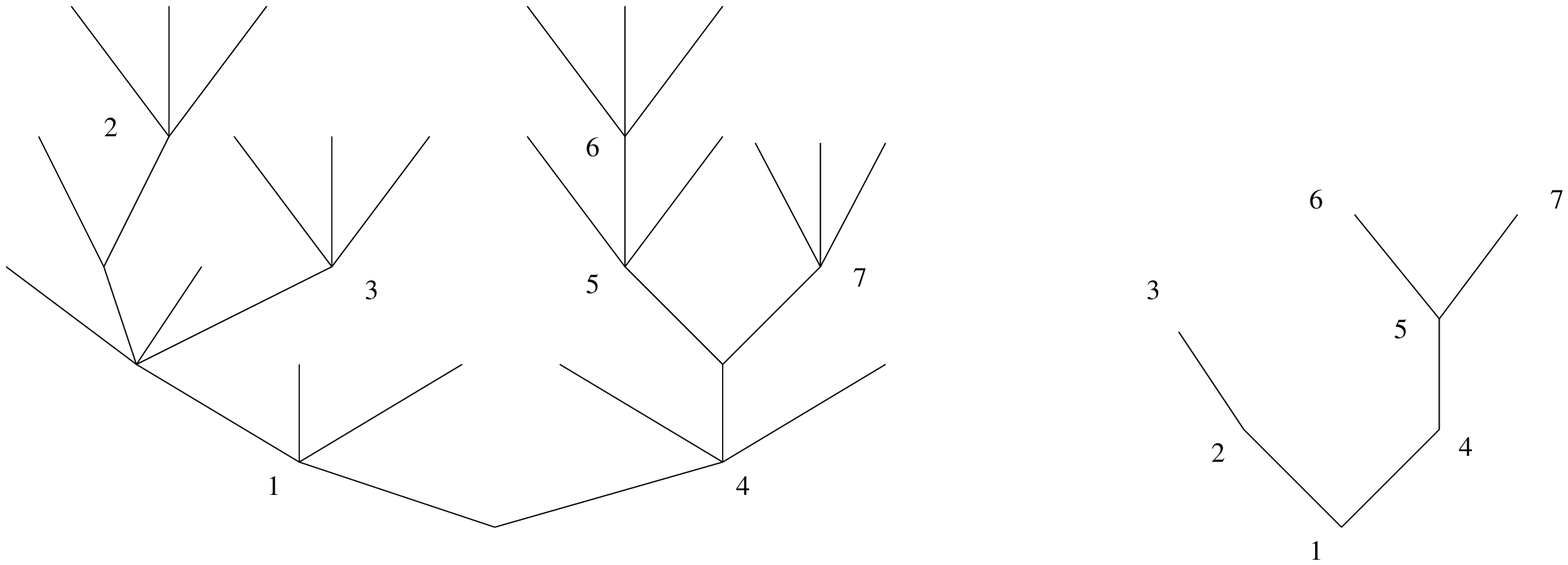}
\caption{left: a tree $\bt$, right: the tree $\bt_\ca$ for $\ca=\{3\}$}\label{fig:tA}
\end{figure}

If $\tau$ is a GW tree with offspring distribution $p$, 
the tree $\tau_\ca$ associated with $\cl_\ca(\tau)$, conditioned on
$\cl_\ca(\tau)\ne\emptyset$,  is then, according
to \cite{r:slmbtgwtcnvodgs} Theorem 6,
 a GW tree  whose offspring distribution
is the law of the random variable $X_\ca$ defined as follows:
\begin{itemize}
\item Let $(X_i,i\ge 1)$ be a sequence of independent random variables
  distributed according to $p$.
\item Let $N=\inf\{k,\ X_k\in\ca\}$ and
  $T=\inf\{k,\ \sum_{i=1}^k(X_i-1)=-1\}$.
\item Let $\tilde X$ be a r.v. distributed as
$$1+\sum_{i=1}^N(X_i-1)$$
conditioned on $N\le T$.
\item Then $X_\ca$ is distributed conditionally given $\{\tilde X=k\}$ as
  a binomial r.v. with parameters $k$ and $q=\P(N\le T)=\P(\cl_\ca(\tau)\ne\emptyset)$.
\end{itemize}
Moreover, as $\tau$ is critical,  $\tau_\ca$ (conditioned on
$\{\cl_\ca(\tau)\ne\emptyset\}$) is also critical, see
\cite{r:slmbtgwtcnvodgs} Lemma 6.

Then, $L_\ca(\tau)$ is just the total progeny of $\tau_\ca$.
Remark that $d_\ca$ is also the span of $X_\ca$. Remark that, if
$0\in\ca$, then $L_\ca(\tau)>0$ and thus $q=1$ and $X_\ca=\tilde
X$. Notice that  we may have
$d_\ca>1$. On the contrary, if $0\not\in\ca$, we have $q<1$ and
therefore $\P(X_\ca=1)>0$. As a consequence, we have $d_\ca=1$.

Consider 
  $n_0=d_\ca$   which gives  
\[
\A_n=\{\bt\in\T;\;   L_\ca(\bt)\in [  n , n+d_\ca)\}.
\]
As $L_\ca(\tau)$ , conditioned on being positive, is distributed as the total size of a critical GW tree,
we deduce  from Subsection \ref{sec:totI}  that 
\begin{equation}\label{eq:LA_1}
  \lim_{n\to+\infty}\frac{\P(L_\ca(\tau)\in[n+1,n+1+d_\ca))}{\P(L_\ca(\tau)\in[n,n+d_\ca))}=1
\end{equation}
 and thus by Theorem \ref{theo:KGen} that \reff{eq:LA_2}
holds. $\square$

\end{proof}

\subsection{The sub-critical case}

Let  $p$  be  an offspring distribution.   Let  $\ca\subset  \N$  such  that
$p(\ca)>0$.  For every  $\theta>0$  such that  $\sum_{k\in \N}  \theta^k
p(k)$ is finite, we define on $\N$ the function $p_\theta$ by
$$\forall k\ge 0,\ p_\theta(k)=\begin{cases}
c_\ca(\theta)\theta^kp(k) & \mbox{if }k\in\ca,\\
\theta^{k-1}p(k) & \mbox{if } k\not\in\ca
\end{cases}$$
where the normalizing constant  $c_\ca(\theta)$ is given by:
\[
c_\ca(\theta)=\frac{1-\sum_{k\not\in
    \ca}\theta^{k-1}p(k)}{\sum_{k\in\ca}\theta^kp(k)}\cdot
\]

We denote by $I$ the set of $\theta$ such that $p_\theta$ defines a probability
distribution on $\N$. Notice that $I$ is an interval with bounds
$\theta_0<1\leq \theta_1$. We have the special cases $\theta_0=0$  if $0\in \ca$
and $\theta_0=p(0)$ if $\ca=\N^*$. 

\begin{prop}
\label{prop:sub}
  Let  $\tau$ be  a GW  tree with offspring distribution   $p$ satisfying
$  p(0)>0 $ and $p(0)+p(1)<1$.   Let  $\ca\subset  \N$  such  that
$p(\ca)>0$. For   every   $\theta\in   I$,   let
  $\tau_\theta$ be a GW tree with offspring distribution  $p_\theta$. Then the
  conditional distributions of $\tau$ given $\{\cl_\ca(\tau)=n\}$ and of
  $\tau_\theta$ given $\{\cl_\ca(\tau_\theta)=n\}$ are the same.
\end{prop}

\begin{rem}
This proposition covers Kennedy's result \cite{k:gwctp} for $\ca=\N$
and the pruning procedure of \cite{adh:pgwttvmp} for $\ca=\{0\}$.
\end{rem}

\begin{proof}
Let $\bt\in\T_0$. Then we have, using the definition of $p_\theta$ and
 \reff{eq:sum_k}:
\begin{align*}
\P(\tau_\theta=\bt) & =\prod _{v\in\bt}p_\theta(k_v(\bt))\\
& = \prod _{v\in \bt, k_v(\bt)\in\ca}
c_\ca(\theta)\theta^{k_v(\bt)}p(k_v(\bt))\prod _{v\in\bt, k_v(\bt)\not\in
  \ca}\theta^{k_v(\bt)-1}p(k_v(\bt))\\
& =
c_\ca(\theta)^{L_\ca(\bt)}\theta^{\sum_{v\in\bt}
  k_v(\bt)-L_{\ca^c}(\bt)}\P(\tau=\bt)\\  
& =c_\ca(\theta)^{L_\ca(\bt)}\theta^{\Card(\bt)-1-L_{\ca^c}(\bt)}\P(\tau=\bt)\\
& =\theta^{-1}(\theta c_\ca(\theta))^{L_\ca(\bt)}\P(\tau=\bt).
\end{align*}
We deduce that
\begin{align*}
\P(L_\ca(\tau_\theta)=n) & =\sum
_{\bt\in\T_0,\ L_\ca(\bt)=n}\P(\tau_\theta=\bt)\\
& =\theta^{-1}(\theta c_\ca(\theta))^{n}\sum
_{\bt\in\T_0,\ L_\ca(\bt)=n}\P(\tau=\bt)\\
& =\theta^{-1}(\theta c_\ca(\theta))^{n}\P(L_\ca(\tau)=n)
\end{align*}
and finally, for every $\bt\in\T_0$ such that $L_\ca(\bt)=n$, we have
\begin{multline*}
\P(\tau_\theta=\bt\bigm|
L_\ca(\tau_\theta)=n)=\frac{\P(\tau_\theta=\bt)}{\P(L_\ca(\tau_\theta)=n)}\\
=\frac{\theta^{-1}(\theta
  c_\ca(\theta))^{n}\P(\tau=\bt)}{\theta^{-1}(\theta
  c_\ca(\theta))^{n}\P(L_\ca(\tau_\theta)=n)}=\P(\tau=\bt\bigm| 
L_\ca(\tau)=n).
\end{multline*}
\end{proof}
We shall say that the offspring distribution $p$ is generic (with respect
to $\ca$)  if there  exists $\theta_c\in I$  such that  $p_{\theta_c}$ is
critical.
\begin{cor}
\label{cor:K-sub}
Let $\tau$  be a  sub-critical GW tree  with offspring  distribution $p$
satisfying Assumption  \reff{eq:cond-p}.  Let $\ca\subset  \N$ such that
$p(\ca)>0$.  For  every $\theta\in  I$, let $\tau_\theta$  be a  GW tree
with  offspring distribution  $p_\theta$.  If  $p$ is  generic,  that is
there exists $\theta_c\in I$ such that $p_{\theta_c}$ is critical, then
\[
\dist(\tau|L_\ca(\tau)=n d_\ca+1
)\underset{n\to+\infty}{\longrightarrow} \dist(\tau_{\theta_c}^*)
\]
and
\[
\dist(\tau|L_\ca(\tau)\geq n 
)\underset{n\to+\infty}{\longrightarrow} \dist(\tau_{\theta_c}^*).
\]
\end{cor}

\begin{rem}
The first convergence of the corollary remains valid for a super-critical
offspring distribution but not the second one as the conditional
distribution cannot be written as a mixture of the first one as the
tree may be infinite.
\end{rem}

\begin{rem}
   \label{rem:condensation}
If the critical value $\theta_c$ of Corollary \ref{cor:K-sub} does not
exist, then we observe a condensation phenomenon: the limiting tree
does not have an infinite spine, but exhibits a unique vertex with an
infinite number of children, see \cite{j:sgtcgwrac} for $\ca=\N$ and the
forthcoming  paper \cite{ad:ncgwtcc} for the general case. 
\end{rem}

\section{Conditioning by the size of a high generation}
\label{sec:geom}

We end this paper with a conditioning which does not enter into the
framework of Theorem \ref{theo:KGen}. However its proof can be easily
adapted. For a tree $\bt$, we denote by
$$\cg_n(\bt)=\Card(\{u\in\bt,\ |u|=n\})$$
the size of the $n$-th generation of $\bt$. Then we have

\begin{prop}\label{prop:gen}
Let  $\tau$ be  a critical GW  tree with  offspring distribution   $p$ satisfying
  Assumption   \reff{eq:cond-p}.  Let $(\alpha_n, n \in \N) $ be a
  sequence of positive
integers. If for all $j\in \N^*$
\begin{equation}\label{eq:condZn}
\lim_{n\to+\infty}\frac{\P(\cg_{n-j}(\tau)=\alpha_n)}{\P(\cg_n(\tau)=\alpha_n)}=1,  
\end{equation}
then we have
\begin{equation}\label{eq:Zn}
\dist(\tau|\cg_n(\tau)=\alpha_n)\underset{n\to+\infty}{\longrightarrow}\dist(\tau^*). 
\end{equation}
\end{prop}

\begin{proof}
For every tree $\bt\in\T_0$, every $x\in\cl_0(\bt)$ and every tree
$\tilde\bt\in\T$, we have
$$\cg_n(\bt\circledast(\tilde\bt,x))=\cg_n(\bt)+\cg_{n-|x|}(\tilde\bt)$$
which generalizes Assumption \reff{eq:A=A+B}.

The same computations as in the proof of Theorem \ref{theo:KGen} give
for $\bt\in\T_0$, $x\in\cl_0(\bt)$ and $n\ge H(\bt)$:
\begin{align*}
\P(\tau\in\T(\bt,x),\cg_n(\tau)=\alpha_n) &
=\frac{1}{p(0)}\P(\tau=\bt)\P(\cg_{n-|x|}(\tau)=\alpha_n-\cg_n(\bt))\\
& =\P(\tau^*\in\T(\bt,x))\P(\cg_{n-|x|}(\tau)=\alpha_n).
\end{align*}

Therefore, we obtain by Assumption \reff{eq:condZn}:
\begin{align*}
\lim_{n\to+\infty}\P(\tau\in
\T(\bt,x)|\cg_n(\tau)=\alpha_n) &
=\lim_{n\to+\infty}\P(\tau^*\in\T(\bt,x))\frac{\P(\cg_{n-|x|}(\tau)=\alpha_n)}{\P(\cg_n(\tau)=\alpha_n)}\\
& =\P(\tau^*\in\T(\bt,x)).
\end{align*}
The result follows from Lemma \ref{lem:cv-determing}. 
\end{proof}

\begin{cor}
\label{cor:geom}
Let  $\tau$ be  a critical GW  tree with offspring distribution   $p$ given
by a mixture of a 
geometric distribution with parameter $q\in(0,1)$ and a Dirac mass
  at 0, i.e. $p(0)=1-q$ and $p(k)=q^2(1-q)^{k-1}$ for $k\ge 1$.
Let $(\alpha_n,n\in \N)$ be a
  sequence of positive
integers such that $\lim_{n\rightarrow+\infty } n^{-2} \alpha_n=0$. Then
we have:
\[
\dist(\tau|\cg_n(\tau)=\alpha_n
)\underset{n\to+\infty}{\longrightarrow}\\dist(\tau^*).
\]
\end{cor}

\begin{proof}
In that particular case, the generating function $\varphi_n$ of
$\cg_n(\tau)$ is explicitly known and we have for every $s\in[0,1]$
\[
\varphi_n(s)=\frac{nc-(nc-1)s}{(nc+1)-ncs}
\]
with $c=(1-q)/q$.
Expanding $\varphi_n$ gives for every $k\ge 1$:
\[
\P(\cg_n(\tau)=k)=\frac{(nc)^{k-1}}{(nc+1)^{k+1}},
\]
and therefore for $j\geq 1$
\[
\lim_{n\to+\infty}\frac{\P(\cg_{n-j}(\tau)=\alpha_n)}{\P(\cg_n(\tau)=\alpha_n)}
=\lim_{n\to+\infty}\frac{n(nc+1)}{(n-j)((n-j)c+1)}
\left(\frac{1+\frac{1}{nc}}{1+\frac{1}{(n-j)c}}\right)^{\alpha_n}=1.
\]
Then use Proposition \ref{prop:gen} to conclude. 
\end{proof}

\begin{rem}
As for Theorem \ref{theo:KGen}, we can obtain the converse of
Proposition \ref{prop:gen}. We deduce that, in the geometric case of
Corollary \ref{cor:geom}, the
GW tree $\tau$ conditioned on $\{\cg_n(\tau)= k\lfloor n^a
\rfloor\}$, with $k\in \N^*$, 
 converges in distribution to Kesten's tree if and only if $a\in
 [0,2)$. 
\end{rem}

Let $X$ be a random variable with distribution $p$, $d$ the span of $X$ and set
$B=\E[X(X-1)]$. We recall the theorem of \cite{nv:lltcwp}. Assume
that $p$ is critical, that Assumption
\reff{eq:cond-p} holds and that $B$ is finite. 
If 
\begin{equation}
   \label{eq:alpha2}
\lim_{n\to+\infty}\alpha_n=+\infty
\quad\text{and}\quad
\limsup_{n\to+\infty}\frac{\alpha_n}{n}<+\infty ,
\end{equation}
then we have:
\[
\lim_{n\rightarrow+\infty } B^2n^2 \left(1+
  \frac{2d}{Bn}\right)^{\alpha_n} \P(\cg_n(\tau)=d\alpha_n)=4d.
\]
We  also recall Theorem 1 of \cite{nv:ltpldgwp}. Let $\rho$ be  the
convergence radius of the generating function of $p$. Assume that $p$ is
critical, that Assumption
\reff{eq:cond-p} holds and that
$\rho>1$. Assume also that
\begin{equation}
   \label{eq:alpha1}
\lim_{n\to+\infty}\frac{\alpha_n}{n}=+\infty
\quad\text{and}\quad
\lim_{n\to+\infty}\frac{\alpha_n}{n^2}=0.
\end{equation}
Then there exists  $c\in \R$ such that:
\[
\lim_{n\rightarrow+\infty } B^2n^2 \expp{ \frac{ 2d\alpha_n}{B n} +
  c\frac{\alpha_n}{n^2}\log(\alpha_n/n)} 
 \P(\cg_n(\tau)=d\alpha_n)=4d.
\]

Then using Proposition \ref{prop:gen}, we give an immediate extension of
Corollary \ref{cor:geom} to a large class of offspring distributions.

\begin{prop}
  Let  $p$ be  a critical  offspring distribution  satisfying Assumption
  \reff{eq:cond-p}  and such  that $B$  is finite.   Assume  either that
  $(\alpha_n,n\in\N)$  is  a sequence  of  positive integers  satisfying
  \reff{eq:alpha2}  or  that   $\rho>1$  and  $(\alpha_n,n\in\N)$  is  a
  sequence of positive  integers satisfying \reff{eq:alpha1}. Let $\tau$
  be a critical GW tree with offspring distribution $p$. Then we have
\[
\dist(\tau|\cg_n(\tau)=d\alpha_n
)\underset{n\to+\infty}{\longrightarrow}\dist(\tau^*).
\]
\end{prop}

\bigskip
{\bf Acknowledgments.} The authors want to thank anonymous referees
for their useful remarks that improved considerably the presentation
of that paper, and in particular for pointing out the references
\cite{ck:rncpccgwta,nv:ltpldgwp,nv:lltcwp}.

\bibliographystyle{abbrv}
\bibliography{biblio}

\end{document}